\RequirePackage{fix-cm}
\documentclass[smallextended]{svjour3}       
\smartqed  
\usepackage{graphicx}
\usepackage[utf8]{inputenc} 
\usepackage{amsmath,amsfonts,amssymb,geometry}
\usepackage{graphicx}
\usepackage{bbold}
\usepackage{cite}
\usepackage{url}
\usepackage{authblk}
\usepackage[colorinlistoftodos]{todonotes}
\usepackage[colorlinks=true, allcolors=blue]{hyperref}
\newtheorem{thm}{Theorem}[section]

 \newtheorem{prop}[thm]{Proposition}
 
 \newtheorem{rem}[thm]{Remark}
\newtheorem{ex}{Example}
\begin{document}

\title{Existence of solutions for a singular double phase in Sobolev-Orlicz spaces with variable exponents in a complete manifold
}

\titlerunning{Existence of solutions for a singular double phase}        

\author{Ahmed Aberqi  \and Jaouad Bennouna \and Omar Benslimane  \and Maria Alessandra Ragusa$^{\ast}$.
}


\institute{Ahmed Aberqi \at
              Laboratory LAMA, Sidi Mohamed Ben Abdellah University, National School of Applied Sciences Fez, Morocco. \\
              \email{aberqi$\_$ahmed@yahoo.fr}           
           \and
             Jaouad Bennouna\at
              Laboratory LAMA, Department of Mathematics, Sidi Mohamed Ben Abdellah University, Faculty of Sciences Dhar El Mahraz, B.P 1796 Atlas Fez, Morocco.\\
              \email{jbennouna@hotmail.com}  
            \and
         Omar Benslimane \at 
         Laboratory LAMA, Department of Mathematics, Sidi Mohamed Ben Abdellah University, Faculty of Sciences Dhar El Mahraz, B.P 1796 Atlas Fez, Morocco. \\
         \email{omar.benslimane@usmba.ac.ma}  
         \and
         Maria Alessandra Ragusa \at   
         Dipartimento di Matematica e Informatica, Universitá di Catania, Viale A.Doria 6, 95125 Catania, Italy.\\
         RUDN University, 6 Miklukho, Maklay St, Moscow, Russia 117198.\\
         \email{mariaalessandra.ragusa@unict.it}
}

\date{Received: date / Accepted: date}

\maketitle

\begin{abstract}
The purpose of this paper is to study a class of double phase problems,  with a singular term and a superlinear parametric term on the right-hand side. Using the method of Nehari manifold combined with the fibering maps, we prove that for all small values of the parameter $\lambda > 0,$ there exist at least two non-trivial positive solutions. Our results extend the previous works Papageorgiou, Repov{\v{s}}, and Vetro \cite{papageorgiou2020positive} and Liu, Dai, Papageorgiou, and Winkert \cite{liu2021existence}, from the case of Musielak-Orlicz Sobolev space, when exponents $p$ and $q$ are constant, to the case of Sobolev-Orlicz spaces with variable exponents in a complete manifold.
\end{abstract}

\section{Introduction}\label{section1}
\label{intro}
Let $ ( M, g ) $ be a smooth complete compact Riemannian N-manifold. In this article, we are addressed in proving the existence of at least two weak solutions to the following singular double phase problem.
$$ ( \mathcal{P} ) \, \begin{cases}
- \, \mbox{div} ( | D \mathrm{u} ( z )|^{p( z ) - 2} D \mathrm{u} ( z ) + \mu ( z ) | D \mathrm{u} ( z ) |^{q( z ) - 2} D \mathrm{u} ( z ) \, ) \\ \hspace*{0.6cm} = \frac{g( z )}{\mathrm{u} ( z )^{\gamma ( z )}} + \lambda \, | \mathrm{u} ( z ) |^{r( z ) - 1}  & \text{in \,\,M }, \\[0.3cm]
\, \mathrm{u} \,  > \, 0  & \text{in \,\,M },\\
\, \mathrm{u} \,  = \, 0  & \text{on \,$\partial$M }.
\end{cases} $$ 
We suppose the subsequent assumptions: 
\begin{itemize}
\item[(i)] The variables exponents $p, q, r \in C( \overline{M} )$ and satisfy the following assumptions: 
\begin{equation}\label{1}
1 <  q^{-} \leq q^{+} < p^{-} \leq p^{+} < r^{-} \leq r^{+} < + \infty,
\end{equation}
where $ q^{+} = \displaystyle \sup_{z \in \overline{M}} q( z ), \,\, q^{-} = \inf_{z \in \overline{M}} q( z ), \,\, p^{+} = \sup_{z \in \overline{M}} p( z ), \,\, p^{-} = \inf_{z \in \overline{M}} p( z ), \\ r^{+} = \sup_{z \in \overline{M}} r( z ), \,\, \mbox{and} \,\, r^{-} = \inf_{z \in \overline{M}} r( z ).$
\begin{equation}\label{2}
\frac{p^{-}}{q^{+}} < 1 + \frac{1}{N}.
\end{equation}
\item[(ii)] $\mu : \overline{M} \rightarrow [1, + \infty )$ is Lipschitz continuous.
\item[(iii)] $ g( z ) \in L^{\infty} ( M )$ and $ g( z ) \geq 0$ for a.a $ z \in M $ with $ g \not\equiv 0$. 
\item[(iv)] $0 <  \gamma ( z ) \in C( \overline{M} ), \,\, 0 < \gamma^{-} \leq \gamma^{+} < 1.$
\item[(v)] $ \lambda$ is a positive parameter.
\end{itemize}

Off late, nonlinear elliptic equations and variational problems involving variable exponents growth conditions have become hugely popular area of investigation owing to its manifold applications including  the study of fluid filtration in porous media, constrained heating, elasto-plasticity, optimal control, financial mathematics and others. Interested readers may refer to \cite{chen2006variable, gwiazda2008non, ruuvzivcka2004modeling, zhikov2004density}  and the references therein for more background of applications. Originally, while studying the behaviour of strongly anisotropic materials, Zhikov discovered that their hardening properties changed radically with the point, which is called the Lavrentiev phenomenon, see for example \cite{zhikov1987averaging, zhikov1995lavrentiev, zhikov1997some, jikov2012homogenization}. In order to describe this phenomenon, he introduced the functional 
\begin{equation}\label{omar}
 \mathrm{u} \longmapsto \int_{\Omega} \big( | D \mathrm{u} |^{p} + \mu ( z ) \, | D \mathrm{u} |^{q} \big) \,\, dx,
\end{equation}
where the integrand switches two different elliptic behaviours. \\

Several interesting works have been carried out on the double phase problem with a Dirichlet boundary condition. For example, Liu and Dai \cite{liu2018existence} studied the following problem 
$$ \begin{cases}
- \, \mbox{div} ( | D \mathrm{u} ( z )|^{p - 2} D \mathrm{u} ( z ) + a ( z ) | D \mathrm{u} ( z ) |^{q - 2} D \mathrm{u} ( z ) \, ) = f( z, \mathrm{u} ( z ) )  & \text{in \,\,$\Omega$ }, \\[0.3cm]
\, \mathrm{u} \,  = \, 0  & \text{on \,$\partial \Omega$ },
\end{cases} $$ 
and proved the existence and the multiplicity of the results, with the sign-changing solutions by variational method. A similar treatment has been recently done by Gasi{\'n}ski and Papageorgiou in \cite{gasinski2019constant} via the Nehari manifold method. Following this direction, the authors in \cite{papageorgiou2020ground} have shown that the double phase problem driven by the sum of the $p$-Laplace operator and a weighted $q$-Laplacian $( q < p )$, with a weight function which is not bounded away from zero and the reaction term, is $( p - 1 )-$superlinear, has a ground state solution of constant sign and a nodal (sign-changing) solution, employing the Nehari method. In particular, using the theory of pseudomonotone operators, Gasi{\'n}ski, and Winkert in \cite{gasinski2020existence} proved the existence and uniqueness of a weak solution to quasilinear elliptic equations with double phase phenomena and a reaction term depending on the gradient, under quite general assumptions on the convection term and towering some linear conditions on the gradient variable. Finally, for a deeper comprehension, we refer the reader to \cite{aberqi2019existence, acerbimingione, benslimane2020existence1, benslimane2020existence, benslimane2020existence2, benslimane2020existence3, benslimane2021existence, duzaarmingione} and the references therein.

Next, we would like to mention some papers of double phase problems with singularity. Papageorgiou, Repov{\v{s}}, and Vetro in \cite{papageorgiou2020positive} studied the existence of positive solutions for a class of double phase Dirichlet equations which has the combined effects of a singular term and of a parametric super-linear term. Readers may refer to \cite{papageorgiou2020positive, liu2021existence} for ideas and techniques developed in order to guarantee the existence of at least two non-trivial positive solutions to double phase problems driven by a singularity.

Considering another novel aspect, we go back to the double phase problem with variable exponents which few authors consider. Ragusa and Tachikawa in \cite{ragusa2019regularity}  are the first ones who have achieved the regularity theory for minimizers of \eqref{omar} with variable exponents. Moreover, in \cite{tachikawa2020boundary} Tachikawa, provides the Hölder continuity up to the boundary of minimizers of so-called double phase functional with variable exponents, under suitable Dirichlet boundary conditions. Recently, Aberqi, Bennouna, Benslimane, and Ragusa in \cite{benslimane2020existence4} studied the existence of non-negative non-trivial solutions for a class of double-phase problems where the source term is a Caratheodory function that satisfies the Ambrosetti-Rabinowitz type condition in the framework of Sobolev-Orlicz spaces with variable exponents in complete manifold, using the Nehari manifold and some variational techniques. Furthermore, they proved the Hölder inequality, continuous and compact embedding. The readers may refer to the work due to \cite{aubin1982nonlinear, benslimane2020existence1, gaczkowski2016sobolev, hebey2000nonlinear, guo2015dirichlet, papageorgiou2019nonlinear} and the references given there.\\

Inspired by the above-mentioned papers, we study the problem $( \mathcal{P} )$ which contains a singular term and a parametric superlinear term in the framework of Sobolev spaces with variable exponents in Complete manifolds. To the best of our knowledge, the results presented here are new. However, we address the challenges presented by the fact that the $p(z)-$ laplacian and $ q( z )-$ laplacian operators possess more complicated nonlinearities than the  $p$-laplacian and $q$-laplacian operator, since $\Delta_{p( z )}$ and $ \Delta_{q( z )}$ are not homogeneous. Moreover, we can not use Lagrange Multiplier Theorem in many problems involving this operator, which shows that our problem has more difficulty than the operators $p$-Laplace type.\\

Now, we define the solution of the problem $( \mathcal{P} )$ in a weaker sense.

\begin{definition}
A function $\mathrm{u} \in W_{0}^{1, q( z )} ( M )$ is said to be a weak solution if $ \mathrm{u} > 0$ for a.a $ z \in M $ and 
\begin{align*}
&\int_{M} \bigg( | D \mathrm{u} ( z ) |^{p( z ) - 2} D \mathrm{u} ( z ) + \mu ( z ) \, | D \mathrm{u} ( z ) |^{q( z ) - 2} D \mathrm{u}( z ) \, \bigg) \, D \varphi ( z ) \,\, dv_{g} ( z )\\& = \int_{M} g( z ) \, | \mathrm{u}  ( z ) |^{- \gamma ( z )} \, \varphi ( z ) \,\, dv_{g} ( z ) + \lambda \int_{M} | \mathrm{u} ( z ) |^{r( z ) -1 } \varphi ( z ) \,\, dv_{g} ( z ),
\end{align*}
is satisfied for all $ \varphi \in \mathcal{D}( M ).$ Where $\mathcal{D}( M )$ denote the space of $C^{\infty}$ functions with compact support in $M.$

\end{definition}

\begin{remark}
For the singular term, the weak solution is by definition a function $ \mathrm{u} \in W_{0}^{1, q( z )} ( M )$ such that $\mathrm{u} ( z )^{- \gamma ( z )} \, \varphi ( z ) \in L^{1} ( M )$ for every $ \varphi \in \mathcal{D}( M )$. Hence, the above definition of a weak solution is well-defined.
\end{remark}

\begin{remark}
In the case when $ q( z ) < r( z ),$ we have the validity of the coercivity properties of the functional energy associated with the problem $( \mathcal{P} )$ ( see Lemma \ref{lemma1} below ) which guaranteed the compactness for sequences with uniformly bounded energy. Unfortunately, in the case when $q( z ) = r( z )$ would cause our coercivity properties to fail. 
\end{remark}

The main result of our paper can be stated as follows.
\begin{theorem} \label{theo1}
Under assumptions (i)-(v), we assume that the complete N-manifold $( M, g )$ has property $B_{vol} ( \lambda, v )$. Then, there exists $\tilde{\lambda}^{*}_{0} > 0$ such as for all $ \lambda \in ( 0, \tilde{\lambda}^{*}_{0} ]$, the problem $( \mathcal{P} )$ has at least two positive solutions $\tilde{\mathrm{u}}^{*}, \, \tilde{v}^{*} \in W_{0}^{1, q( z )} ( M )$ such as $ J_{\lambda} ( \tilde{\mathrm{u}}^{*} ) < 0 \leq J_{\lambda} ( \tilde{v}^{*} ).$
\end{theorem}

The context of the paper is organized as follows. In section \ref{sec2} we will recall the definitions and some properties of Sobolev spaces with variable exponents and Sobolev spaces with variable exponents in complete manifolds. The readers can consult the following papers \cite{aubin1982nonlinear, benslimane2020existence1, benslimane2020existence, benslimane2020existence2, benslimane2020existence3, hebey2000nonlinear, trudinger1968remarks} for details. In section \ref{sec3}, we introduce the Nehari manifold associated with $( \mathcal{P} )$ and we study three parts, corresponding to local minima, local maxima and the points of inflection. Moreover, we show the existence of two non-trivial positive solutions when the parameter $ \lambda > 0$ is sufficiently small.

\section{Notations and Basic Properties}\label{sec2}
In order to discuss the problem $( \mathcal{P} )$, we need some facts on spaces $W_{0}^{1, q( x )} ( \Omega ) $ where $\Omega$ is an open subset of $ \mathbb{R}^{N}$ and $ W_{0}^{1, q( x )} ( M )$ which are called the Sobolev spaces with variable exponents and the Sobolev spaces with variable exponents in complete manifolds setting. For this reason, we will recall some properties involving the above spaces, which can be found in \cite{aubin1982nonlinear, benslimane2020existence1, fan2001spaces, gaczkowski2016sobolev, hebey2000nonlinear, guo2015dirichlet} and references therein.
\subsection{Sobolev spaces with variable exponents}
Given $ \Omega \subset \mathbb{R}^{N} $ a bounded open domain, with $ N \geq 2.$\\
Let $q(\cdot): \Omega \rightarrow (1, \infty )$ be a measurable function, we define the real numbers $q^{+}$ and $q^{-}$ as follows
$$q^{+} = ess \,sup \{ \, q( z );\, z \in \Omega \, \} \,\, \mbox{and} \,\, q^{-} = ess\, inf \{ \, q( z );\, z \in \Omega \, \}.$$
\begin{definition} \cite{fan2001spaces}
We define the Lebesgue space with variable exponent $ L^{q(\cdot)} ( \Omega )$ as follows
$$L^{q(\cdot)} ( \Omega ) = \bigg \{ \mathrm{u} : \Omega \rightarrow \mathbb{R}; \, \rho _{q(\cdot)} ( \mathrm{u} ) = \int_{\Omega} |\, \mathrm{u}( z )\,|^{q( z )} \,\, dz < + \infty  \, \bigg \}, $$
endowed with the Luxemburg norm $$ ||\, \mathrm{u} \,||_{q(\cdot)} = \inf \bigg\{ \, \mu > 0: \, \rho_{q(\cdot)} \bigg( \frac{\mathrm{u}}{\mu} \bigg) \leq 1 \, \bigg\}, $$ 
if $ q^{+} < + \infty.$
\end{definition}
\begin{prop} \cite{fan2001spaces}
The space $ ( L^{q(\cdot)} ( \Omega ), \, ||\,.\,||_{q(\cdot)} )$ is a separable Banach space, and uniformly convex for $ 1 < q^{-} \leq q^{+} < +\infty,$  hence reflexive.
\end{prop}
\begin{prop} \cite{fan2001spaces} (Hölder inequality)
$$ \bigg|\, \int_{\Omega} \mathrm{u}\, v \,\, dx \, \bigg| \leq \bigg( \, \frac{1}{q^{-}} + \frac{1}{( q^{'} )^{-}} \bigg) \, ||\, \mathrm{u}\,||_{q(\cdot)} ||\, v\,||_{q^{'} (\cdot)}, \hspace*{0.5cm} \forall \mathrm{u}, \,v \in L^{q(\cdot)} ( \Omega ) \times L^{q^{'}(\cdot)} ( \mathcal{Q} ),$$
with $ \frac{1}{q( z )} + \frac{1}{q^{'} ( z )} = 1.$
\end{prop}
\begin{definition} \cite{fan2001spaces}
We define the variable exponents Sobolev space by
$$ W^{1, q( z )} ( \Omega ) = \big\{ \, \mathrm{u} \in L^{q( z )} ( \Omega ) \,\, \mbox{and} \,\, |\, D \mathrm{u}\,| \in L^{q( z )} ( \Omega ) \,\big\},$$
with the norm $$ ||\, \mathrm{u}\,||_{W^{1, q( z )} ( \Omega )} = ||\, \mathrm{u}\,||_{L^{q( z )} ( \Omega )} + ||\, D \mathrm{u} \||_{L^{q( z )} ( \Omega )}, \,\, \forall \mathrm{u} \in W^{1, q( z )} ( \Omega ).$$
\end{definition}
\begin{rem} \cite{fan2001spaces}
$$ W^{1, q( z )}_{0} ( \Omega ) = \overline{C^{\infty}_{0} ( \Omega )}^{ W^{1, q( z )} ( \Omega )}.$$
\end{rem}

\subsection{Sobolev spaces with variable exponents in complete manifolds}
\begin{definition} \cite{hebey2000nonlinear}
Let $ ( M, g ) $ be a smooth Riemannain N-manifolds and let $ D $ be the Levi-Civita connection. If $\mathrm{u} $ is a smooth function on $M$, then $ D^{k} \mathrm{u}  $ denotes the $k-$th covariant derivative of $\mathrm{u} $, and $ | \, D^{k} \mathrm{u}  \, | $ the norm of $ D^{k} \mathrm{u}  $ defined in local coordinates by
$$ | \, D^{k} \mathrm{u}  \, |^{2} = g^{i_{1} j_{1}} \cdots g^{i_{k} j_{k}} \, ( D^{k} \mathrm{u}  )_{i_{1} \cdots i_{k}} \, ( D^{k} \mathrm{u}  )_{j_{1} \cdots j_{k}} $$
where Einstein's convention is used.
\end{definition}
\begin{definition} \cite{hebey2000nonlinear}
To define variable Sobolev spaces, given a variable exponent $q$ in $ \mathcal{P} ( M ) $ ( the set of all measurable functions $p(\cdot) : M \rightarrow [ 1, \infty ]$ ) and a natural number $k$, introduce 
$$ C^{q(\cdot)}_{k} ( M ) = \{ \, \mathrm{u}  \in C^{\infty} ( M ) \,\, \mbox{such that } \,\, \forall j \,\, 0 \leq j \leq k \,\, | \, D^{k} \mathrm{u}  \, | \in L^{q( \cdot ) } ( M )\, \}. $$
On $ C^{q( \cdot )}_{k} ( M ) $ define the norm 
$$ || \, \mathrm{u}  \, ||_{L^{q( \cdot )}_{k}} = \sum_{j = 0}^{k} || \, D^{j} \mathrm{u}  \, ||_{L^{q( \cdot )}}. $$
\end{definition}
\begin{definition} \cite{gaczkowski2016sobolev}
The Sobolev spaces $ L_{k}^{q( \cdot )} ( M ) $ is the completion of $ C^{q(\cdot)}_{k} ( M ) $ with respect to the norm $ || \, \mathrm{u}  \, ||_{L^{q( \cdot )}_{k}}$. If $\Omega$ is a subset of $M$, then $L^{q( \cdot )}_{k, 0} ( \Omega )$ is the completion of $C^{q( \cdot )}_{k} ( M ) \cap C_{0} ( M )$ with respect to $ || \, \cdot \,||_{L^{q( \cdot )}_{k}},$ where $C_{0} ( \Omega )$ denotes the vector space of continuous functions whose support is a compact subset of $\Omega.$
\end{definition}
\begin{definition}\cite{hebey2000nonlinear}
Given $ ( M, g ) $ a smooth Riemannian manifold, and $ \gamma : \, [\, a, \, b \, ] \longrightarrow M $ a curve of class $ C^{1} $. The length of $ \gamma $ is 
$$ \ell( \gamma ) = \int_{a}^{b} \bigg( \, g \, \bigg( \, \frac{d \gamma }{d t }, \, \frac{d \gamma}{d t}\, \bigg) \, \bigg)^{\frac{1}{2}} \,\,dt, $$
and for a pair of points $ z, \, y \in M$, we define the distance $ d_{g} ( z, y ) $ between $z$ and $y$ by 
$$ d_{g} ( z, y ) = \inf \, \{ \, \ell( \gamma ) : \, \gamma: \, [ \, a, \, b \,] \rightarrow M \,\, \mbox{such that} \,\, \gamma ( a ) = z \,\, \mbox{and} \,\, \gamma ( b ) = y \, \}. $$
\end{definition}
\begin{definition} \cite{hebey2000nonlinear}
A function $ s: \, M \longrightarrow \mathbb{R} $ is log-Hölder continuous if there exists a constant $c$ such that for every pair of points $ \{ z, \, y \} $ in $ M$ we have
$$ | \, s( z ) - s( y ) \, | \leq c\,.\, \bigg( \log ( e + \frac{1}{d_{g} ( z, y )} \, ) \bigg)^{-1}. $$
We note by $ \mathcal{P}^{log} ( M ) $ the set of log-Hölder continuous variable exponents. 
The link with $ \mathcal{P}^{log} ( M ) $ and $ \mathcal{P}^{log} ( \mathbb{R}^{N} ) $ is given by the following proposition:
\end{definition}
\begin{proposition} \cite{aubin1982nonlinear, gaczkowski2016sobolev}
Let $ q \in \mathcal{P}^{log} ( M ) $, and let $ ( \Omega, \phi ) $ be a chart such that 
$$ \frac{1}{2} \delta_{i j } \leq g_{i j} \leq 2 \, \delta_{i j } $$
as bilinear forms, where $ \delta_{i j} $ is the delta Kronecker symbol. Then $ qo\phi^{-1} \in \mathcal{P}^{log} ( \phi ( \Omega ) ).$
\end{proposition}
\begin{definition} \cite{hebey2000nonlinear}
We say that the N-manifold $ ( M, g ) $ has property $ B_{vol} ( \lambda, v ),$ where $\lambda$ is a constant, if its geometry is bounded in the following sense:\\
$ \hspace*{1cm} \bullet \,\, \mbox{The Ricci tensor of g noted by Rc ( g ) verify,} \,\,Rc ( g ) \geq \lambda ( N - 1 ) \, g $ for some $ \lambda,$ where $N$ is the dimension of $M.$\\
$ \hspace*{1cm} \bullet  $ There exists some $ v > 0 $ such that $ | \, B_{1} ( z ) \, |_{g} \geq v \,\, \forall z \in M,$ where $B_{1} ( z ) $ are the balls of radius 1 centered at some point $z$ in terms of the volume of smaller concentric balls.
\end{definition}
\begin{proposition} \cite{aubin1982nonlinear,hebey2000nonlinear} \label{prop2}
Let $ ( M, g ) $ be a complete Riemannian N-manifold. Then, if the embedding $ L^{1}_{1} ( M ) \hookrightarrow L^{\frac{n}{N - 1}} ( M )$ holds, then whenever the real numbers $q$ and $p$ satisfy $$ 1 \leq q < N, $$ and $$ q \leq p \leq q* = \frac{N q}{N - q}, $$ the embedding $ L^{q}_{1} ( M ) \hookrightarrow L^{p} ( M ) $ also holds.
\end{proposition}\label{prop5}
\begin{proposition} \cite{aubin1982nonlinear,hebey2000nonlinear}\label{prop3}
Assume that the smooth complete compact Riemannian N-manifold $ ( M, g ) $ has property $ B_{vol} ( \lambda, v ) $ for some $ ( \lambda, v ).$ Then there exist positive constants $ \delta_{0} = \delta_{0} ( N, \, \lambda, \, v ) $ and $ A = A ( N, \, \lambda, \, v ) $, we have, if $ R \leq \delta_{0} $, if $ z \in M $ if $ 1 \leq q \leq N $, and if $ \mathrm{u}  \in L^{q}_{1,0} ( \, B_{R} ( z ) \, ) $ the estimate 
$$ || \, \mathrm{u}  \, ||_{L^{p}} \leq A \,p \, || \, D \mathrm{u}  \, ||_{L^{q}},$$ where $ \frac{1}{p} = \frac{1}{q} - \frac{1}{N}.$ 
\end{proposition}
\begin{proposition} \cite{aubin1982nonlinear,hebey2000nonlinear, gaczkowski2016sobolev} \label{prop6}
Assume that for some $ ( \lambda, v ) $ the smooth complete compact Riemannian N-manifold $( M, g ) $ has property $ B_{vol} ( \lambda, v ) $. Let $ p \in \mathcal{P} ( M ) $ be uniformly continuous with $ q^{+} < N.$ Then $ L^{q( \cdot )}_{1} ( M ) \hookrightarrow L^{p( \cdot )} ( M ) \,\, \forall q \in \mathcal{P} ( M ) $ such that $ q \ll p \ll q* = \frac{N q}{N - q}. $ In fact, for $ || \, \mathrm{u}  \, ||_{L^{q( \cdot )}_{1}} $ sufficiently small we have the estimate $$ \rho_{p( \cdot )} ( \mathrm{u}  ) \leq G \, ( \, \rho_{q( \cdot )} ( \mathrm{u}  ) + \rho_{q( \cdot )} ( | \, D \mathrm{u}  \, | ) \, ), $$ where the positive constant $G$ depend on $ N, \, \lambda, \, v, \, q $ and $ p $.
\end{proposition}
\begin{proposition} \cite{guo2015dirichlet} \label{prop7}
Let $ \mathrm{u}  \in L^{q( z )} ( M ), \,\{ \, \mathrm{u} _{k} \,\} \subset L^{q( z )} ( M ), \, k \in \mathbb{N},$ then we have 
\begin{enumerate}
\item[(i)]  $|| \mathrm{u}  ||_{q( z )} < 1 \,\,\mbox{( resp. = 1, $>$ 1 )} \iff \rho_{q( z )} ( \mathrm{u}  ) < 1 \,\,\mbox{( resp. = 1, $>$ 1 )},$
\item[(ii)]For $ \mathrm{u}  \in L^{q( z )} ( M ) \backslash \{ 0 \}, \,\, || \mathrm{u}  ||_{q( z )} = \lambda \Longleftrightarrow \rho_{q( z )} \big( \frac{\mathrm{u} }{\lambda} \big) = 1.$
\item[(iii)]  $ || \mathrm{u}  ||_{q( z )} < 1 \Rightarrow || \mathrm{u}  ||_{q( z )}^{q^{+}} \leq \rho_{q( z )} ( \mathrm{u}  ) \leq || \mathrm{u}  ||_{q( z )}^{q^{-}},$
\item[(iv)]  $ || \mathrm{u}  ||_{q( z )} > 1 \Rightarrow || \mathrm{u}  ||_{q( z )}^{q^{-}} \leq \rho_{q( z )} ( \mathrm{u}  ) \leq || \mathrm{u}  ||_{q( z )}^{q^{+}},$
\item[(v)] $ \lim_{k \rightarrow + \infty} || \mathrm{u} _{k} - \mathrm{u}  ||_{q( z )} = 0 \iff \lim_{k \rightarrow + \infty} \rho_{q( z )} ( \mathrm{u} _{k} - \mathrm{u}  ) = 0. $
\end{enumerate}
\end{proposition}
\begin{definition} \cite{guo2015dirichlet}
The Sobolev space $ W^{1, q( z )} ( M )$ consists of such functions $ \mathrm{u}  \in L^{q( z )} ( M )$ for which $ D^{k} \mathrm{u}  \in L^{q( z )} ( M )$ $k = 1, 2, \cdots, n.$ The norm is defined by 
$$ ||\, \mathrm{u} \,||_{W^{1, q( z )} ( M )} = ||\, \mathrm{u}  \,||_{L^{q( z )} ( M )} + \sum_{k = 1}^{n} ||\, D^{k} \mathrm{u}  \,||_{L^{q( z )} ( M )}.$$ 
The space $ W_{0}^{1, q( z )} ( M )$ is defined as the closure of $ C^{\infty} ( M ) $ in $ W^{1, q( z )} ( M ).$
\end{definition}
\begin{theorem}\label{theo2} \cite{benslimane2020existence4}
Let $M$ be a compact Riemannian manifold with a smooth boundary or without boundary and $ q( z ), \, p( z ) \in C( \overline{M} ) \cap L^{\infty} ( M ).$ Assume that $$ q( z ) < N , \hspace*{0.5cm} p( z ) < \frac{N\, q( z )}{N - q( z )} \,\, \mbox{for} \,\, z \in \overline{M}.$$
Then, $$ W^{1, q( z )} ( M ) \hookrightarrow L^{p( z )} ( M ),$$
is a continuous and compact embedding.
\end{theorem}
\begin{theorem}\label{theo3} 
Let $M$ be a compact Riemannian manifold with a smooth boundary or without boundary and $ q( z ), \, r( z ) \in C( \overline{M} ) \cap L^{\infty} ( M ).$ Assume that $$ q( z ) < N , \hspace*{0.5cm} r( z ) < \frac{N\, q( z )}{N - q( z )} \,\, \mbox{for} \,\, z \in \overline{M}.$$
Then, $$ W^{1, q( z )} ( M ) \hookrightarrow L^{r( z )} ( M ),$$
is a continuous and compact embedding.
\end{theorem}
\begin{proof}
The demonstration of this theorem is the same as the previous one.
\end{proof} 
\begin{proposition} \cite{aubin1982nonlinear}
If $( M, g )$ is complete, then  $W^{1, q( z )} ( M ) = W^{1, q( z )}_{0} ( M ).$
\end{proposition}
The weighted variable exponent Lebesgue space $L_{\mu( z )}^{q( z )} ( M )$ is defined as follows:
$$L_{\mu( z )}^{q( z )} ( M ) = \{ \mathrm{u} : M \rightarrow \mathbb{R} \,\, \mbox{is measurable such that}, \int_{M} \mu( z ) \, | \mathrm{u}  ( z ) |^{q( z )} \,\, dv_{g} ( z ) < + \infty \},$$
with the norm $$ || \mathrm{u}  ||_{q( z ), \mu( z )} = \inf \{ \alpha > 0; \int_{M} \mu ( z ) \, \bigg| \frac{\mathrm{u} ( z )}{\alpha} \bigg|^{q( z )} \,\, dv_{g} ( z ) \leq 1 \,\}.$$
Moreover, the weighted modular on $L_{\mu ( z )}^{q( z )} ( M )$ is the mapping $ \rho_{q( \cdot ), \mu ( \cdot )} : L_{\mu ( z )}^{q( z )} ( M ) \rightarrow \mathbb{R}$ defined like
$$ \rho_{q( \cdot ), \mu ( \cdot )} ( \mathrm{u}  ) = \int_{M} \mu ( z ) | \mathrm{u} ( z )|^{q( z )} \,\, dv_{g} ( z ).$$
\begin{ex} 
As a simple example of $\mu ( z ),$ we can take $ \mu ( z ) = ( 1 + | z | )^{\varepsilon ( z )} $ with $\varepsilon ( \cdot ) \in C_{+} ( \overline{M} ),$
\end{ex}
where $ C_{+} ( \overline{M} ) = \{ q/ q \in C( \overline{M} )\,\, \mbox{with} \,\, q( z ) > 1 \,\, \mbox{for} \,\, z \in \overline{M} \}.$
\section{Existence results:}\label{sec3}
In this section, we are going to prove our main result stated as Theorem \ref{theo1} in section \ref{section1}. For this, we recall that $\mathrm{ J}_{\lambda} : W_{0}^{1, q( z )} ( M ) \rightarrow \mathbb{R}$ is the corresponding energy function for the problem $( \mathcal{P} )$ given by:
\begin{align*}
\mathrm{ J}_{\lambda} ( \mathrm{u} ) = & \int_{M}\frac{1}{p( z )}\, | D \mathrm{u} ( z ) |^{p( z )} \,\, dv_{g} ( z ) + \int_{M} \frac{\mu ( z )}{q( z )}\, | D \mathrm{u} ( z ) |^{q( z )} \,\, dv_{g} ( z ) \\&-  \int_{M}\frac{1}{1 - \gamma ( z )} \,g( z )\, | \mathrm{u} ( z ) |^{1 - \gamma ( z )} \,\, dv_{g} ( z ) - \int_{M} \frac{\lambda}{r( z )} \,| \mathrm{u} ( z ) |^{r( z )} \,\, dv_{g} ( z ).
\end{align*}
From the presence of the singular term $g( z ) \, | \mathrm{u} ( z ) |^{1 - \gamma ( z )},$ we known that $\mathrm{ J}_{\lambda}$ is not $C^{1}.$ In order to overcome this, we will make use the Nehari manifold corresponding to the functional $\mathrm{ J}_{\lambda} $ defined as follows
\begin{align*}
\mathcal{N}_{\lambda} = \{  \mathrm{u} \in W_{0}^{1, q( z )} ( M ) \backslash  \{ 0\}; \,&|| D \mathrm{u} ||_{p( z )}^{p( z )} + \int_{M} \mu ( z )\, | D \mathrm{u} ( z ) |^{q( z )} \,\, dv_{g} ( x ) \\&= \int_{M} g( z ) | \mathrm{u} ( z ) |^{1 - \gamma ( z )} \,\, dv_{g} ( z ) + \lambda || \mathrm{u} ||_{r( z )}^{r( z )} \,\}.
\end{align*}
It is easy to see that $\mathcal{N}_{\lambda}$ is smaller than $ W_{0}^{1, q( z )} ( M )$ and it contains the weak solutions of problem $( \mathcal{P} )$. The functional $\mathrm{ J}_{\lambda} \bigg|_{\mathcal{N}_{\lambda}}$ can have nice properties which fail to be true globally.\\
For further considerations we decompose the set $\mathcal{N}_{\lambda}$ into three disjoint parts:
\begin{align*}
&\mathcal{N}_{\lambda}^{+} = \{ \mathrm{u} \in \mathcal{N}_{\lambda} : \int_{M} ( p( z ) - \gamma ( z ) -1 ) \, | D \mathrm{u} ( z ) |^{p( z )} \,\, dv_{g} ( z ) \\&+  \int_{M} ( q( z ) + \gamma ( z ) - 1 ) \, \mu ( z )\, | D \mathrm{u} ( z ) |^{q( z )} \,\,dv_{g} ( z ) - \lambda \int_{M} ( r( z ) + \gamma ( z ) - 1 ) | \mathrm{u} ( z ) |^{r( z )} \,\, dv_{g} ( z ) > 0 \,\},
\end{align*}
\begin{align*}
&\mathcal{N}_{\lambda}^{0} = \{ \mathrm{u} \in \mathcal{N}_{\lambda} : \int_{M} ( p( z ) - \gamma ( z ) -1 ) \, | D \mathrm{u} ( z ) |^{p( z )} \,\, dv_{g} ( z ) \\&+ \int_{M} ( q( z ) + \gamma ( z ) - 1 ) \, \mu ( z )\, | D \mathrm{u} ( z ) |^{q( z )} \,\,dv_{g} ( z ) = \lambda \int_{M} ( r( z ) + \gamma ( z ) - 1 ) | \mathrm{u} ( z ) |^{r( z )} \,\, dv_{g} ( z ) \,\},
\end{align*}
\begin{align*}
&\mathcal{N}_{\lambda}^{-} = \{ \mathrm{u} \in \mathcal{N}_{\lambda} : \int_{M} ( p( z ) - \gamma ( z ) -1 ) \,| D \mathrm{u} ( z ) |^{p( z )} \,\, dv_{g} ( z ) \\&+ \int_{M} ( q( z ) + \gamma ( z ) - 1 ) \,\mu ( z )\, | D \mathrm{u} ( z ) |^{q( z )} \,\,dv_{g} ( z ) - \lambda \int_{M} ( r( z ) + \gamma ( z ) - 1 ) | \mathrm{u} ( z ) |^{r( z )} \,\, dv_{g} ( z )  < 0 \,\}.
\end{align*}
Note that, since $ \mu ( \cdot ) : \overline{M} \longrightarrow [ 1,\, + \infty ),$ then, there exists $\mu_{0} > 0,$ and for all $ z \in M,$ we have that $ \mu ( z ) > \mu_{0}.$\\
We start with the following Lemma about the coercivity of the energy functional $\mathrm{ J}_{\lambda} \bigg|_{\mathcal{N}_{\lambda}}$.
\begin{lemma}\label{lemma1}
If hypotheses (i)-(v) hold. Then, $\mathrm{ J}_{\lambda} \bigg|_{\mathcal{N}_{\lambda}}$ is coercive.
\end{lemma}
\begin{proof}
Let $ \mathrm{u} \in \mathcal{N}_{\lambda} $ with $|| \mathrm{u} || > 1,$ where $||\, \cdot\,||$ is the induced norm of $W^{1, q( z )}_{0} ( M ) \backslash \{ 0\}.$ From the definition of the Nehari manifold $\mathcal{N}_{\lambda}$ we have 
\begin{align*}
\lambda \int_{M} | \mathrm{u} ( z ) |^{r( z )} \,\, dv_{g} ( z ) = &\int_{M} | D \mathrm{u} ( z ) |^{p( z )} \,\, dv_{g} ( z ) + \int_{M} \mu ( z ) \, | D \mathrm{u} ( z ) |^{q( z )} \,\,dv_{g} ( z ) \\&- \int_{M} g( z ) \, | \mathrm{u} ( z ) |^{1 - \gamma ( z )} \,\, dv_{g} ( z ). 
\end{align*}
Hence, according to \eqref{1}, propositions \ref{prop5}, \ref{prop7}, Theorems \ref{theo2}, \ref{theo3}  and based on inequality 2.3 in \cite{benslimane2020existence4} we obtain
\begin{align*}
\mathrm{ J}_{\lambda} ( \mathrm{u} )& = \int_{M} \frac{1}{p( z )} | D \mathrm{u} ( z ) |^{p( z )} \,\, dv_{g} ( z ) + \int_{M} \frac{\mu( z )}{q( z )} | D \mathrm{u} ( z ) |^{q( z )} \,\, dv_{g} ( z )\\& \hspace*{0.3cm} - \int_{M} \frac{1}{1 - \gamma ( z )} \, g( z ) \, | \mathrm{u} ( z ) |^{1 - \gamma ( z )} \,\, dv_{g} ( z ) - \lambda \, \int_{M} \frac{1}{r( z )} | \mathrm{u} ( z ) |^{r( z )} \,\, dv_{g} ( z ) \\& \geq \frac{1}{p^{+}} \int_{M} | D \mathrm{u} ( z ) |^{p( z )} \,\, dv_{g} ( z ) + \frac{1}{q^{+}} \int_{M} \mu ( z ) | D \mathrm{u} ( z ) |^{q( z )} \,\, dv_{g} ( z )\\& \hspace*{0.3cm} - \frac{1}{1 - \gamma^{+}} \int_{M} g( z ) | \mathrm{u} ( z ) |^{1 - \gamma ( z )} \,\, dv_{g} ( z ) - \frac{\lambda}{r^{-}} \int_{M} | \mathrm{u} ( z ) |^{r( z )} \,\, dv_{g} ( z )\\& = \frac{1}{p^{+}} \int_{M} | D \mathrm{u} ( z ) |^{p( z )} \,\, dv_{g} ( z ) + \frac{1}{q^{+}} \int_{M} \mu ( z ) | D \mathrm{u} ( z ) |^{q( z )} \,\, dv_{g} ( z )\\& \hspace*{0.3cm} - \frac{1}{1 - \gamma^{+}} \int_{M} g( z ) | \mathrm{u} ( z ) |^{1 - \gamma ( z )} \,\, dv_{g} ( z ) - \frac{1}{r^{-}} \bigg[ \int_{M} | D \mathrm{u} ( z )|^{p( z )} \,\,dv_{g} ( z ) \\& \hspace*{0.3cm} + \int_{M} \mu( z ) |D \mathrm{u} ( z )|^{q( z )} \,\,dv_{g} ( z ) - \int_{M} g( z ) \, | \mathrm{u} ( z ) |^{1 - \gamma ( z )} \,\, dv_{g} ( z ) \, \bigg] \\&\geq \bigg( \frac{1}{p^{+}} - \frac{1}{r^{-}} \bigg) \int_{M} | D \mathrm{u} ( z ) |^{p( z )} \,\, dv_{g} ( z ) + \mu_{0} \, \bigg( \frac{1}{q^{+}} - \frac{1}{r^{-}} \bigg) \int_{M} | D \mathrm{u} ( z ) |^{q( z )} \,\, dv_{g} ( z ) \\& \hspace*{0.3cm} + \bigg( \frac{1}{r^{-}} - \frac{1}{1 - \gamma^{+}} \bigg) \int_{M} g( z ) \, | \mathrm{u} ( z ) |^{1 - \gamma ( z )} \,\, dv_{g} ( z ) \\& \geq \bigg[ \frac{1}{c} \,\bigg( \frac{1}{p^{+}} - \frac{1}{r^{-}} \bigg) + \frac{\mu_{0}}{D^{p^{+}} (c + 1)p^{+}} \, \bigg( \frac{1}{q^{+}} - \frac{1}{r^{-}} \bigg) \bigg] \, \rho_{p( . )} ( \mathrm{u} )\\& \hspace*{0.3cm}+ \bigg( \frac{1}{r^{-}} - \frac{1}{1 - \gamma^{+}} \bigg) \int_{M} g( z )\,| \mathrm{u} ( z ) |^{1 - \gamma ( z )} \,\, dv_{g} ( z ) \\& \geq c_{1} \, || \mathrm{u} ||^{p^{+}} + c_{2} \, || \mathrm{u} ||^{1 - \gamma^{+}},
\end{align*}
for some $c_{1}, c_{2} > 0 $ ( since $ q^{+} < p^{+} < r^{-} $ ), and $c$ is the Poincaré constant.\\
From that, since $ 0 < \gamma^{+} < 1$ and $ 1 - \gamma^{+} < 1 < p^{+},$ we conclude that $\mathrm{ J}_{\lambda} \bigg|_{\mathcal{N}_{\lambda}}$ is coercive. 
\end{proof}
Let $ \sigma_{\lambda}^{+} = \inf_{\mathcal{N}_{\lambda}} \mathrm{ J}_{\lambda}.$
\begin{lemma}\label{lemma2}
If the assumptions (i)-(v) are satisfied and $ \mathcal{N}_{\lambda}^{+} \neq 0.$ Then $\sigma_{\lambda}^{+} < 0.$
\end{lemma}
\begin{proof}
Suppose that $ \mathrm{u} \in \mathcal{N}_{\lambda}^{+},$ hence from the definition of $ \mathcal{N}_{\lambda}^{+}$ we have
$$ \lambda || \mathrm{u} ||_{r( z )}^{r( z )} < \frac{p^{+} + \gamma^{+} - 1}{r^{-} + \gamma^{-} - 1} || D \mathrm{u} ||_{p( z )}^{p( z )} + \frac{q^{+} + \gamma^{+} - 1}{r^{+} - \gamma^{-} - 1} \int_{M} \mu ( z ) | D \mathrm{u} ( z ) |^{q( z )} \,\, dv_{g} ( z ).$$
And since $ \mathcal{N}_{\lambda}^{+} \subset \mathcal{N}_{\lambda}, $ we get
\begin{align*}
\mathrm{ J}_{\lambda} ( \mathrm{u} ) &= \int_{M} \frac{1}{p( z )}\, | D \mathrm{u} ( z ) |^{p( z )} \,\, dv_{g} ( z ) +  \int_{M} \frac{\mu ( z )}{q( z )}\, | D \mathrm{u} ( z ) |^{q( z )} \,\, dv_{g} ( z ) \\& \hspace*{0.3cm} - \int_{M} \frac{g( z )}{1 - \gamma ( z )} \,| \mathrm{u} ( z ) |^{1 - \gamma ( z )} \,\, dv_{g} ( z ) - \lambda \int_{M} \frac{1}{r( z )}  \,| \mathrm{u} ( z ) |^{r( z )} \,\, dv_{g} ( z ) \\& \leq \frac{1}{p^{-}} \int_{M} | D \mathrm{u} ( z ) |^{p( z )} \,\, dv_{g} ( z ) + \frac{1}{q^{-}} \int_{M} \mu ( z ) \, | D \mathrm{u} ( z ) |^{q( z )} \,\, dv_{g} ( z )\\& \hspace*{0.3cm} - \frac{1}{1 - \gamma^{+}} \int_{M} g( z )\,| \mathrm{u} ( z ) |^{1 - \gamma( z )} \,\, dv_{g} ( z ) - \frac{\lambda}{r^{+}} \int_{M} | \mathrm{u} ( z ) |^{r( z )} \,\, dv_{g} ( z )\\& = \frac{1}{p^{-}} \int_{M} | D \mathrm{u} ( z ) |^{p( z )} \,\, dv_{g} ( z ) + \frac{1}{q^{-}} \int_{M} \mu( z ) \, | D \mathrm{u} ( z ) |^{q( z )} \,\, dv_{g} ( z )\\& \hspace*{0.3cm} - \frac{1}{1 - \gamma^{+}} \bigg[ \int_{M} | D \mathrm{u} ( z ) |^{p( z )} \,\, dv_{g} ( z ) + \int_{M} \mu( z ) \, |D \mathrm{u}( z )|^{q( z )} \,\, dv_{g} ( z ) - \lambda \int_{M} | \mathrm{u} ( z ) |^{r( z )} \,\, dv_{g} ( z ) \bigg]\\&\hspace*{0.3cm} - \frac{\lambda}{r^{+}} \int_{M} | \mathrm{u} ( z ) |^{r( z )} \,\, dv_{g} ( z ) \\& = \bigg( \frac{1}{p^{-}} - \frac{1}{1 - \gamma^{+}} \bigg) \, \int_{M} | D \mathrm{u} ( z )|^{p( z )} \,\, dv_{g} ( z ) + \bigg( \frac{1}{q^{-}} - \frac{1}{1 - \gamma^{+}} \bigg) \, \int_{M} \mu( z ) \, | D \mathrm{u} ( z ) |^{q( z )} \,\, dv_{g} ( z ) \\& \hspace*{0.3cm} + \lambda \, \bigg( \frac{1}{1 - \gamma^{+}} - \frac{1}{r^{+}} \bigg) \, \int_{M} | \mathrm{u} ( z ) |^{r( z )} \,\, dv_{g} ( z ) \\& \leq \bigg[ \bigg( \frac{1}{p^{-}} - \frac{1}{1 - \gamma^{+}} \bigg) + \bigg( \frac{1}{1 - \gamma^{+}} - \frac{1}{r^{+}}  \bigg) \frac{p^{+} + \gamma^{+} - 1}{r^{-} + \gamma^{-} - 1} \bigg] \int_{M} |D \mathrm{u} ( z ) |^{p( z )} \,\,dv_{g} ( z ) \\& \hspace*{0.3cm} + \bigg[ \bigg( \frac{1}{q^{-}} - \frac{1}{1 - \gamma^{+}} \bigg) + \bigg( \frac{1}{1 - \gamma^{+}} - \frac{1}{r^{+}} \bigg) \frac{q^{+} + \gamma^{+} - 1}{r^{-} + \gamma^{-} - 1}  \bigg] \int_{M} \mu ( z )\,| D \mathrm{u}( z )|^{q( z )} \,\, dv_{g} ( z ) \\& \leq \frac{p^{+} + \gamma^{+} - 1}{1 - \gamma^{-}} \bigg[ \frac{1}{r^{+}} - \frac{1}{p^{-}} \bigg] \int_{M} | D \mathrm{u} ( z ) |^{p( z )} \,\, dv_{g} ( z )\\& \hspace*{0.3cm} + \mu_{0}\, \frac{q^{+} + \gamma^{+} - 1}{1 - \gamma^{-}} \bigg[ \frac{1}{r^{+}} - \frac{1}{q^{-}} \bigg] \int_{M} | D \mathrm{u} ( z ) |^{q( z )} \,\,dv_{g} ( z ) \\& < 0 \hspace*{1cm} \mbox{Since $ q^{-} < p^{-} < r^{+}.$}
\end{align*}
Which implies that $$\mathrm{ J}_{\lambda} \bigg|_{\mathcal{N}_{\lambda}} < 0.$$
Hence, $$ \sigma_{\lambda}^{+} < 0.$$
\end{proof}
\begin{lemma}\label{lemma3}
Under assumptions (i)-(v), there exists $\lambda^{*} > 0$ such as $ \mathcal{N}_{\lambda}^{0} = \emptyset$ for all $ \lambda \in ( 0, \lambda^{*} ).$
\end{lemma}
\begin{proof}
Suppose otherwise, that is $ \mathcal{N}_{\lambda}^{0} \neq \emptyset$ for all $ \lambda > 0,$ we can find $ \mathrm{u} \in \mathcal{N}_{\lambda}^{0}$ such as 
\begin{align} \label{3}
&( p^{+} + \gamma^{+} - 1) \int_{M} | D \mathrm{u} ( z )|^{p( z )} \,\, dv_{g} ( z ) + ( q^{+} + \gamma^{+} - 1) \int_{M} \mu ( z ) \, | D \mathrm{u} ( z ) |^{q( z )} \,\, dv_{g} ( z )\nonumber \\ & = \lambda ( r^{-} + \gamma^{-} - 1 ) \int_{M} | \mathrm{u} ( z ) |^{r( z )} \,\, dv_{g} ( z ),  
\end{align}
since $ \mathrm{u} \in \mathcal{N}_{\lambda},$ we also have
\begin{align} \label{4}
&( r^{+} + \gamma^{+} - 1) \int_{M} | D \mathrm{u} ( z ) |^{p( z )} \,\, dv_{g} ( z ) + ( r^{+} + \gamma^{+} - 1) \int_{M} \mu ( z ) | D \mathrm{u} ( z ) |^{q( z )} \,\, dv_{g} ( z )\nonumber \\& = ( r^{-} + \gamma^{-} - 1) \int_{M} g( z ) | \mathrm{u} ( z ) |^{1 - \gamma ( z )} \,\, dv_{g} ( z ) + \lambda ( r^{-} + \gamma^{-} - 1) \int_{M} | \mathrm{u} ( z ) |^{r( z )} \,\,dv_{g} ( z ).
\end{align}
Subtracting \eqref{3} from \eqref{4} yields
\begin{align}\label{5}
&( p^{+} - r^{+} ) \int_{M} | D \mathrm{u} ( z ) |^{p( z )} \,\, dv_{g} ( z ) + ( q^{+} - r^{+} ) \int_{M} \mu ( z ) | D \mathrm{u} ( z ) |^{q( z )} \,\, dv_{g} ( z ) \nonumber \\& = ( r^{-} + \gamma^{-} - 1) \int_{M} g( z ) | \mathrm{u} ( z ) |^{1 - \gamma ( z )} \,\,dv_{g} ( z ).
\end{align}
According to Theorems 13.17 of Hewitt-Stromberg [\cite{hewitt2013real}, p. 196], Theorems \ref{theo2}, \ref{theo3} and propositions \ref{prop5}, \ref{prop7}, we deduce from \eqref{5} that
 $$ \min \{ || \mathrm{u} ||^{p^{+}}, \, || \mathrm{u} ||^{q^{+}} \, \} \leq c_{3} \, || \mathrm{u} ||^{1 - \gamma^{-}} \,\, \mbox{for some} \,\, c_{3} > 0. \,\, \mbox{Since} \,\, 1 - \gamma^{-} < q^{+} < p^{+} < r^{-}.$$
Thus, 
\begin{equation}\label{said}
|| \mathrm{u} || \leq c_{4} \,\, \mbox{for some} \,\, c_{4} > 0.
\end{equation}
From \eqref{3} and Theorems \ref{theo2}, we have 
$$ || \mathrm{u} ||^{r^{-}} \geq \frac{1}{\lambda \, c_{5}} \, \min \{ || \mathrm{u} ||^{p^{+}}, \, || \mathrm{u} ||^{q^{+}} \, \}  \,\,\mbox{for some} \,\, c_{5} > 0,$$
which implies that $$ || \mathrm{u} || \geq \bigg( \frac{1}{\lambda \, c_{5}} \bigg)^{\frac{1}{r^{-} - p^{+}}} \hspace*{1cm} \mbox{or} \hspace*{1cm} || \mathrm{u} || \geq \bigg( \frac{1}{\lambda \, c_{5}} \bigg)^{\frac{1}{r^{-} - q^{+}}}.$$
If $ \lambda \rightarrow 0^{+} $ due to $ q^{+} < p^{+} < r^{-},$ then $|| \mathrm{u} || \longrightarrow + \infty.$ Which contradicts \eqref{said}.
\end{proof}
\begin{lemma}\label{lemma4}
Under assumptions (i)-(v), there exists $\tilde{\lambda}^{*} \in (0, \lambda^{*} ] $ such that $ \mathcal{N}_{\lambda}^{\pm} \neq \emptyset$ for all $\lambda \in ( 0, \tilde{\lambda}^{*} ).$ In addition, for all $ \lambda \in ( 0, \tilde{\lambda}^{*} ),$ there exists $ \tilde{\mathrm{u}}^{*} \in \mathcal{N}_{\lambda}^{+} $ such that $ \mathrm{ J}_{\lambda} ( \tilde{\mathrm{u}}^{*} ) = \sigma_{\lambda}^{+} < 0 $ and $ \tilde{\mathrm{u}}^{*} ( z ) \geq 0 $ for a.a $ z \in M.$
\end{lemma}
\begin{proof}
Let $ \mathrm{u} \in W_{0}^{1, q( z )} ( M ) \backslash \{ 0\}$ and consider the function $\zeta_{\mathrm{u}} : ( 0, + \infty ) \rightarrow \mathbb{R} $ defined by $$ \zeta_{\mathrm{u}} ( t ) = t^{p^{+} - r^{-}} \int_{M} | D \mathrm{u} ( z ) |^{p( z )} \,\, dv_{g} ( z ) - t^{-r^{-} - \gamma^{-} + 1} \int_{M} g( z ) | \mathrm{u} ( z ) |^{1 - \gamma ( z )} \,\, dv_{g} ( z ).$$
Since $ r^{-} - p^{+} < r^{-} + \gamma^{-} - 1 $ we can find $ \tilde{t}_{0} > 0$ such as $$ \zeta_{\mathrm{u}} ( \tilde{t}_{0} ) = \max_{t > 0} \zeta_{\mathrm{u}} ( t ).$$
Then, we have 
\begin{align*}
\zeta'_{\mathrm{u}} ( \tilde{t}_{0} ) = 0 & \Rightarrow ( p^{+} - r^{-} ) \, \tilde{t}_{0}^{p^{+} - r^{-} - 1} \int_{M} | D \mathrm{u} ( z ) |^{p( z )} \,\, dv_{g} ( x ) \\& \hspace*{0.3cm} + ( r^{-} + \gamma^{-} - 1 ) \, \tilde{t}_{0}^{-r^{-} - \gamma^{-}} \int_{M} g( z )| \mathrm{u} ( z )|^{1 - \gamma ( z )} \,\, dv_{g} ( z ) = 0 \\& \Rightarrow \tilde{t}_{0} = \bigg( \frac{( r^{-} + \gamma^{-} - 1 ) \, \displaystyle \int_{M} g( z ) | \mathrm{u} ( z ) |^{1 - \gamma ( z )} \,\, dv_{g} ( z )}{( r^{-} - p^{+} ) \, \displaystyle \int_{M} | D \mathrm{u} ( z ) |^{p( z )} \,\, dv_{g} ( z )} \bigg)^{\frac{1}{p^{+} + \gamma^{-} - 1}}.
\end{align*}
Furthermore, we have
\begin{align} \label{6}
\zeta_{\mathrm{u}} ( \tilde{t}_{0} ) &= \frac{\bigg[ ( r^{-} - p^{+} ) \, \displaystyle\int_{M} | D \mathrm{u} ( z ) |^{p( z )} \,\, dv_{g} ( z ) \bigg]^{\frac{r^{-} - p^{+}}{p^{+} + \gamma^{-} - 1}}}{\bigg[ ( r^{-} + \gamma^{-} - 1 ) \displaystyle\int_{M} g( z ) | \mathrm{u} ( z ) |^{1 - \gamma ( z )} \,\, dv_{g} ( z ) \bigg]^{\frac{r^{-} - p^{+}}{p^{+} + \gamma^{-} - 1}}} \int_{M} | D \mathrm{u} ( z ) |^{p( z )} \,\,dv_{g} ( z ) \nonumber \\& \hspace*{0.1cm} -\frac{\bigg[ ( r^{-} - p^{+} ) \, \displaystyle\int_{M} | D \mathrm{u} ( z ) |^{p( z )} \,\, dv_{g} ( z ) \bigg]^{\frac{r^{-} + \gamma^{-} - 1}{p^{+} + \gamma^{-} - 1}}}{\bigg[ ( r^{-} + \gamma^{-} - 1 ) \displaystyle \int_{M} g( z ) | \mathrm{u} ( z ) |^{1 - \gamma ( z )} \,\, dv_{g} ( z ) \bigg]^{\frac{r^{-} + \gamma^{-} - 1}{p^{+} + \gamma^{-} - 1}}} \int_{M} g( z ) | \mathrm{u} ( z ) |^{1 - \gamma ( z )} \,\, dv_{g} ( z ) \nonumber \\& = \frac{( r^{-} - p^{+} )^{\frac{r^{-} - p^{+}}{p^{+} + \gamma^{-} - 1}} \, \bigg( \displaystyle \int_{M} | D \mathrm{u} ( z ) |^{p( z )} \,\, dv_{g} ( z ) \bigg)^{\frac{r^{-} + \gamma^{-} - 1}{p^{+} + \gamma^{-} - 1}}}{( r^{-} + \gamma^{-} - 1 )^{\frac{r^{-} - p^{+}}{p^{+} + \gamma^{-} - 1}} \, \bigg( \displaystyle \int_{M} g( z ) | \mathrm{u} ( z ) |^{1 - \gamma ( z )} \,\, dv_{g} ( z ) \bigg)^{\frac{r^{-} - p^{+}}{p^{+} + \gamma^{-} - 1}}} \nonumber \\& - \frac{( r^{-} - p^{+} )^{\frac{r^{-} + \gamma^{-} - 1}{p^{+} + \gamma^{-} - 1}} \, \bigg( \displaystyle \int_{M} | D \mathrm{u} ( z ) |^{p( z )} \,\, dv_{g} ( z ) \bigg)^{\frac{r^{-} + \gamma^{-} - 1}{p^{+} + \gamma^{-} - 1}}}{( r^{-} + \gamma^{-} - 1 )^{\frac{r^{-} + \gamma^{-} - 1}{p^{+} + \gamma^{-} - 1}} \, \bigg( \displaystyle \int_{M} g( z ) | \mathrm{u} ( z ) |^{1 - \gamma ( z )} \bigg)^{\frac{r^{-} - p^{+}}{p^{+} + \gamma^{-} - 1}}} \nonumber \\& = \bigg[ \frac{r^{-} + \gamma^{-} - 1}{r^{-} - p^{+}} \,.\, \frac{(r^{-} - p^{+})^{\frac{r^{-} + \gamma^{-} - 1}{p^{+} + \gamma^{-} - 1}}}{( r^{-} + \gamma^{-} - 1)^{\frac{r^{-} + \gamma^{-} - 1}{p^{+} + \gamma^{-} - 1}}} - \frac{(r^{-} - p^{+})^{\frac{r^{-} + \gamma^{-} - 1}{p^{+} + \gamma^{-} - 1}}}{( r^{-} + \gamma^{-} - 1)^{\frac{r^{-} + \gamma^{-} - 1}{p^{+} + \gamma^{-} - 1}}} \, \bigg] \nonumber \\& \hspace*{0.3cm} \times \frac{\bigg( \displaystyle \int_{M} | D \mathrm{u} ( z ) |^{p( z )} \,\, dv_{g} ( z ) \bigg)^{\frac{r^{-} + \gamma^{-} - 1}{p^{+} + \gamma^{-} - 1}}}{\bigg( \displaystyle \int_{M} g( z ) | \mathrm{u} ( z ) |^{1 - \gamma( z )} \,\, dv_{g} ( z ) \bigg)^{\frac{r^{-} - p^{+}}{p^{+} + \gamma^{-} - 1}}} \nonumber \\& = \frac{p^{+} + \gamma^{-} - 1}{r^{-} - p^{+}} \, \bigg( \frac{r^{-} - p^{+}}{r^{-} + \gamma^{-} - 1} \bigg)^{\frac{r^{-} + \gamma^{-} - 1}{p^{+} + \gamma^{-} - 1}} \, \frac{\bigg( \displaystyle \int_{M} | D \mathrm{u} ( z ) |^{p( z )} \,\, dv_{g} ( z ) \bigg)^{\frac{r^{-} + \gamma^{-} - 1}{p^{+} + \gamma^{-} - 1}}}{\bigg( \displaystyle \int_{M} g( z ) | \mathrm{u} ( z ) |^{1 - \gamma( z )} \,\, dv_{g} ( z ) \bigg)^{\frac{r^{-} - p^{+}}{p^{+} + \gamma^{-} - 1}}}.
\end{align}
According to Proposition \ref{prop7} and if $G$ denotes the best Sobolev constant, we have
\begin{equation}\label{7}
G \, || \mathrm{u} ||^{p^{+}} \leq \int_{M} | D \mathrm{u} ( z ) |^{p( z )} \,\, dv_{g} ( z ).
\end{equation}
Moreover, we have
\begin{equation}\label{8}
\int_{M} g( z ) | \mathrm{u} ( z ) |^{1 - \gamma ( z )} \,\, dv_{g} ( z ) \leq c_{6} \, || \mathrm{u} ||^{1 - \gamma^{-}} \,\, \mbox{for some} \,\, c_{6} > 0.
\end{equation}
According to \eqref{6}, \eqref{7} and \eqref{8} we have
\begin{align*}
&\zeta_{\mathrm{u}} ( \tilde{t}_{0} ) - \lambda \int_{M} | \mathrm{u} ( z ) |^{r( z )} \,\, dv_{g} ( z )\\& = \frac{p^{+} + \gamma^{-} - 1}{r^{-} - p^{+}} \, \bigg( \frac{r^{-} - p^{+}}{r^{-} + \gamma^{-} - 1} \bigg)^{\frac{r^{-} + \gamma^{-} - 1}{p^{+} + \gamma^{-} - 1}} \, \frac{\bigg( \displaystyle \int_{M} | D \mathrm{u} ( z ) |^{p( z )} \,\, dv_{g} ( z ) \bigg)^{\frac{r^{-} + \gamma^{-} - 1}{p^{+} + \gamma^{-} - 1}}}{\bigg( \displaystyle \int_{M} g( z ) | \mathrm{u} ( z ) |^{1 - \gamma( z )} \,\, dv_{g} ( z ) \bigg)^{\frac{r^{-} - p^{+}}{p^{+} + \gamma^{-} - 1}}}\\& \hspace*{0.3cm} - \lambda \, \int_{M} | \mathrm{u} ( z ) |^{r( z )} \,\, dv_{g} ( z ) \\& \geq \frac{p^{+} + \gamma^{-} - 1}{r^{-} - p^{+}} \, \bigg( \frac{r^{-} - p^{+}}{r^{-} + \gamma^{-} - 1} \bigg)^{\frac{r^{-} + \gamma^{-} - 1}{p^{+} + \gamma^{-} - 1}} \, \frac{G^{\frac{r^{-} + \gamma^{-} - 1}{p^{+} + \gamma^{-} - 1}} || u ||^{\frac{p^{+} ( r^{-} + \gamma^{-} - 1 )}{p^{+} + \gamma^{-} - 1}}}{\bigg( c_{6} || \mathrm{u} ||^{1 - \gamma^{-}} \bigg)^{\frac{r^{-} - p^{+}}{p^{+} + \gamma^{-} - 1}}} \\& \hspace*{0.3cm} + \lambda \, c_{7} \, || \mathrm{u} ||^{r^{+}} \\& \geq ( c_{6} - \lambda c_{7} ) \, || \mathrm{u} ||^{r^{+}} \,\,\, \mbox{for some} \,\,\, c_{6}, \, c_{7} > 0.
\end{align*}
Hence, there exists $ \tilde{\lambda}^{*} \in (0, \lambda^{*}]$ independent of $\mathrm{u}$ such as 
\begin{equation}\label{9}
\zeta_{\mathrm{u}} ( \tilde{t}_{0} ) - \lambda \int_{M} | \mathrm{u} ( z ) |^{r( z )} \,\, dv_{g} ( z ) > 0 \,\,\,\mbox{for all} \,\,\,\, \lambda \in (0, \tilde{\lambda}^{*}].
\end{equation}
Now, for $ \mathrm{u} \in W_{0}^{1, q( z )} ( M )$ we consider the function $ \tilde{\zeta}_{z} : ( 0, + \infty ) \rightarrow \mathbb{R} $ defined by 
\begin{align*}
\tilde{\zeta}_{\mathrm{u}} ( t ) = & \, t^{p^{+} - r^{-}} \int_{M} | D \mathrm{u} ( z ) |^{p( z )} \,\, dv_{g} ( z ) + t^{q^{+} - r^{-}} \int_{M} \mu ( z ) \, | D \mathrm{u} ( z ) |^{q( z )} \,\, dv_{g} ( z ) \\& - t^{-r^{-} - \gamma^{-} + 1} \int_{M} g( z ) | \mathrm{u}( z ) |^{1 - \gamma ( z )} \,\, dv_{g} ( z ) \,\,\, \mbox{for all} \,\,\, t > 0.
\end{align*}
According to \eqref{1}, we get that $$ r^{-} - p^{+} < r^{-} - q^{+} < r^{-} + \gamma^{-} - 1.$$
Thus, we can find $ t_{0} > 0$ such as $$ \tilde{\zeta}_{\mathrm{u}} ( t_{0} ) = \max_{t > 0} \tilde{\zeta}_{\mathrm{u}} ( t ).$$
Obviously, we have $ \tilde{\zeta}_{\mathrm{u}} \geq \zeta_{\mathrm{u}} $ and by \eqref{9} we find the existence of $ \tilde{\lambda}^{*} \in (0, \lambda^{*}]$ independent of $u$ such as 
$$ \tilde{\zeta}_{\mathrm{u}} ( t_{0} ) - \lambda \int_{M} | \mathrm{u} ( z ) |^{r( z )} \,\, dv_{g} ( z ) > 0 \,\,\, \mbox{for all} \,\,\, \lambda \in (0, \tilde{\lambda}^{*}].$$
Therefore, there exists $ t_{1} < t_{0} < t_{2}$ such that
\begin{equation}\label{10}
\tilde{\zeta}_{\mathrm{u}} ( t_{1} ) = \lambda \int_{M} | \mathrm{u} ( z ) |^{r( z )} \,\, dv_{g} ( z ) = \tilde{\zeta}_{\mathrm{u}} ( t_{2} ),
\end{equation}
and 
\begin{equation}\label{11}
\tilde{\zeta'}_{\mathrm{u}} ( t_{2} ) < 0 <  \tilde{\zeta'}_{\mathrm{u}} ( t_{1} ),
\end{equation}
where 
\begin{align}\label{12}
\tilde{\zeta'}_{\mathrm{u}} ( t ) &= ( p^{+} - r^{-} ) \,  t^{p^{+} - r^{-} - 1} \int_{M} | D \mathrm{u} ( z ) |^{p( z )} \,\, dv_{g} ( z ) \nonumber \\& \hspace*{0.3cm}+ ( q^{+} - r^{-} ) \, t^{q^{+} - r^{-} - 1} \int_{M} \mu ( z ) \, | D \mathrm{u} ( z ) |^{q( z )} \,\, dv_{g} ( z ) \nonumber \\& \hspace*{0.3cm}- ( - r^{-} - \gamma^{-} + 1 ) t^{-r^{-} - \gamma^{-}} \int_{M} g( z ) | \mathrm{u} ( z ) |^{1 - \gamma ( z )} \,\, dv_{g} ( z ).
\end{align}
Now, we consider the fibering function $\psi_{\mathrm{u}} : [0, + \infty ) \rightarrow \mathbb{R}$ defined by $ \psi_{\mathrm{u}} ( t ) = \mathrm{ J}_{\lambda} ( t \mathrm{u} )$ for all $ t \geq 0.$\\
Since $ \psi_{\mathrm{u}} ( t ) \in C^{2} ( (0, \infty ) ).$ We deduce that
\begin{align*}
\psi'_{\mathrm{u}} ( t_{1} ) =& \, t^{p^{+} - 1}_{1} \int_{M} | D \mathrm{u} ( z ) |^{p( z )} \,\,dv_{g} ( z ) + t_{1}^{q^{+} - 1} \int_{M} \mu ( z )\,| D \mathrm{u} ( z ) |^{q( z )} \,\, dv_{g} ( z ) \\&- t_{1}^{- \gamma^{-}} \int_{M} g( z ) | \mathrm{u} ( z ) |^{1 - \gamma ( z )} \,\, dv_{g} ( z ) - \lambda \, t_{1}^{r^{-} - 1} \int_{M} | \mathrm{u} ( z ) |^{r( z )} \,\,dv_{g} ( z ),
\end{align*}
 and
\begin{align}\label{13}
\psi''_{\mathrm{u}} ( t_{1} ) =& ( p^{+} - 1 ) \, t_{1}^{p^{+} -2} \int_{M} | D \mathrm{u} ( z ) |^{p( z )} \,\, dv_{g} ( z ) + ( q^{+} - 1 )\, t_{1}^{q^{+} - 2} \int_{M} \mu ( z )\,| D \mathrm{u} ( z ) |^{q( z )} \,\, dv_{g} ( z ) \nonumber \\&+ \gamma^{-} \, t^{- \gamma^{-} - 1}_{1} \int_{M} g( z ) | \mathrm{u} ( z ) |^{1 - \gamma ( z )} \,\, dv_{g} ( z ) - \lambda ( r^{-} - 1 ) \, t_{1}^{r^{-} - 2} \int_{M} | \mathrm{u} ( z ) |^{r( z )} \,\, dv_{g} ( z ).
 \end{align}
 According to \eqref{10} and \eqref{11}, we get
 \begin{align*}
& t_{1}^{p^{+} - r^{-}} \int_{M} | D \mathrm{u} ( z ) |^{p( z )} \,\, dv_{g} ( z ) + t_{1}^{q^{+} - r^{-}} \int_{M} \mu ( z ) \,| D \mathrm{u} ( z ) |^{q( z )} \,\, dv_{g} ( z )\\& - t_{1}^{-r^{-} - \gamma^{-} + 1} \int_{M} g( z ) | \mathrm{u} ( z ) |^{1 - \gamma ( z )} \,\, dv_{g} ( z ) = \lambda \int_{M} | \mathrm{u} ( z )|^{r( z )} \,\, dv_{g} ( z ), 
 \end{align*}
which implies by multiplying with $ \gamma^{-} t_{1}^{r^{-} - 2} $ and $ - ( r^{-} - 1 ) \, t_{1}^{r^{-} - 2},$ respectively, that
\begin{align}\label{14}
&\gamma^{-} t_{1}^{p^{+} - 2} \int_{M} | D \mathrm{u} ( z ) |^{p( z )} \,\, dv_{g} ( z ) + \gamma^{-} t_{1}^{q^{+} - 2} \int_{M} \mu ( z ) \,| D \mathrm{u} ( z ) |^{q( z )} \,\, dv_{g} ( z )\nonumber \\& - \gamma^{-} \lambda t_{1}^{r^{-} - 2} \int_{M} | \mathrm{u} ( z ) |^{r( z )} \,\, dv_{g} ( z ) = \gamma^{-} t_{1}^{- \gamma^{-} - 1} \int_{M} g( z ) | \mathrm{u} ( z ) |^{1 - \gamma ( z )} \,\, dv_{g} ( z ),
\end{align}
and 
\begin{align}\label{15}
&- ( r^{-} - 1 ) t_{1}^{p^{+} - 2} \int_{M} | D \mathrm{u} ( z ) |^{p( z )} \,\, dv_{g} ( z ) - ( r^{-} - 1 ) t_{1}^{q^{+} - 2} \int_{M} \mu ( z ) \,| D \mathrm{u} ( z ) |^{q( z )} \,\, dv_{g} ( z ) \nonumber \\&+ ( r^{-} - 1 ) t_{1}^{- \gamma^{-} - 1} \int_{M} g( z ) | \mathrm{u} ( z ) |^{1 - \gamma ( z )} \,\, dv_{g} ( z ) = -\lambda ( r^{-} - 1 ) t_{1}^{r^{-} - 2} \int_{M} | \mathrm{u} ( z ) |^{r( z )} \,\, dv_{g} ( z ).
\end{align}
Applying \eqref{14} in \eqref{13} we obtain
\begin{align}\label{16}
\psi''_{\mathrm{u}} ( t_{1} ) &= ( p^{+} + \gamma^{-} - 1 ) t_{1}^{p^{+} - 2} \int_{M} | D \mathrm{u} ( z ) |^{p( z )} \,\, dv_{g} ( z ) \nonumber \\& \hspace*{0.3cm}+ ( q^{+} + \gamma^{-} - 1 ) t_{1}^{q^{+} - 2} \int_{M} \mu ( z ) \,| D \mathrm{u} ( z ) |^{q( z )} \,\, dv_{g} ( z ) \nonumber \\& \hspace*{0.3cm}- \lambda\, ( r^{-} + \gamma^{-} - 1 ) t_{1}^{- \gamma^{-} - 1} \int_{M} g( z ) | \mathrm{u} ( z ) |^{1 - \gamma ( z )} \,\, dv_{g} ( z ). 
\end{align}
On the other hand, by using the same technique, we apply \eqref{15} in \eqref{13}, we deduce that
\begin{equation} \label{166}
\psi''_{\mathrm{u}} ( t_{1} ) = t_{1}^{1 - r^{-}} \zeta'_{\mathrm{u}} ( t_{1} ) > 0.
\end{equation}
From \eqref{16} and \eqref{166} we conclude that
\begin{align*}
\psi''_{\mathrm{u}} ( t_{1} ) =& ( p^{+} + \gamma^{-} - 1 ) t_{1}^{p^{+}} \int_{M} | D \mathrm{u} ( z ) |^{p( z )} \,\, dv_{g} ( z ) \\&+ ( q^{+} + \gamma^{-} - 1 ) t_{1}^{q^{+}} \int_{M} \mu ( z )\,| D \mathrm{u} ( z ) |^{q( z )} \,\, dv_{g} ( z ) \\&- \lambda ( r^{-} + \gamma^{-} - 1 ) t_{1}^{r^{-}} \int_{M} | \mathrm{u} ( z ) |^{r( z )} \,\, dv_{g} ( z ).
\end{align*}
Thus, $$ t_{1} \mathrm{u} \in \mathcal{N}_{\lambda}^{+} \,\,\, \mbox{for all} \,\,\, \lambda \in ( 0, \tilde{\lambda}^{*} ].$$
Hence, $$ \mathcal{N}_{\lambda}^{+} \neq \emptyset.$$
Using the same method for the point $t_{2}$ and according to \eqref{10} and \eqref{11}, we can see that $\mathcal{N}_{\lambda}^{-} \neq \emptyset.$\\
This shows the first assertion of the proposition. Now consider a minimizer sequence $\{ \mathrm{u}_{m} \}_{m \in \mathbb{N}} \subset \mathcal{N}_{\lambda}^{+}$ such as
\begin{equation}\label{17}
J_{\lambda} ( \mathrm{u}_{m} ) \searrow \sigma_{\lambda}^{+} < 0 \,\,\,\, \mbox{as} \,\,\,\, m \rightarrow \infty.
\end{equation}
According to the fact that $ \mathcal{N}_{\lambda}^{+} \subseteq \mathcal{N}_{\lambda}$ and Lemma \ref{lemma1}, we have that 
$$ \{ \mathrm{u}_{m} \}_{m \in \mathbb{N}} \subseteq W_{0}^{1, q( z )} ( M ) \,\,\,\, \mbox{is bounded}.$$
Therefore, we may assume that
\begin{equation}\label{18}
\mathrm{u}_{m} \rightharpoonup \tilde{\mathrm{u}}^{*} \,\,\,\mbox{in} \,\,\, W_{0}^{1, q( z )} ( M ) \,\,\,\, \mbox{and} \,\,\,\, \mathrm{u}_{m} \rightarrow \tilde{\mathrm{u}}^{*} \,\,\,\, \mbox{in} \,\,\,\, L^{r( z )} ( M ).
\end{equation}
From \eqref{17} and \eqref{18} we know that $$ \mathrm{ J}_{\lambda} ( \tilde{\mathrm{u}}^{*} ) \leq \lim_{m \rightarrow + \infty} \inf \, \mathrm{ J}_{\lambda} ( \mathrm{u}_{m} ) < 0 = \mathrm{ J}_{\lambda} ( 0 ).$$
Hence, $$ \tilde{\mathrm{u}}^{*} \neq 0.$$
Arguing by contradiction, suppose that $ \mathrm{u}_{m} \not \rightarrow \tilde{\mathrm{u}}^{*}$ in $W_{0}^{1, q( z )} ( M ).$ Then we will have
\begin{equation}\label{19}
\lim_{m \rightarrow + \infty} \inf \int_{M} | D \mathrm{u}_{m} ( z ) |^{p( z )} \,\, dv_{g} ( z ) > \int_{M} | D \tilde{\mathrm{u}}^{*} ( z ) |^{p( z )} \,\, dv_{g} ( z ).
\end{equation} 
Thus, using \eqref{19}, we have 
\begin{align} \label{20}
\lim_{m \rightarrow + \infty} \inf \psi'_{\mathrm{u}_{m}} ( t_{1} ) &= \lim_{m \rightarrow + \infty} \inf \bigg[ t_{1}^{p^{+} - 1} \int_{M} | D \mathrm{u}_{m} ( z ) |^{p( z )} \,\, dv_{g} ( z ) \nonumber \\& \hspace*{0.2cm}+ t_{1}^{q^{+} - 1} \int_{M} \mu ( z ) \,| D \mathrm{u}_{m} ( z ) |^{q( z )} \,\, dv_{g} ( z ) - t_{1}^{-\gamma^{-}} \int_{M} g( z ) | \mathrm{u}_{m} ( z ) |^{1 - \gamma ( z )} \,\, dv_{g} ( z )\nonumber \\& \hspace*{0.2cm} - \lambda t_{1}^{r^{-} - 1} \int_{M} | \mathrm{u}_{m} ( z ) |^{r( z )} \,\, dv_{g} ( z ) \bigg] \nonumber \\& > t_{1}^{p^{+} - 1} \int_{M} | D \tilde{\mathrm{u}}^{*} ( z ) |^{p( z )} \,\, dv_{g} ( z ) + t_{1}^{q^{+} - 1} \int_{M} \mu ( z )\, | D \tilde{\mathrm{u}}^{*} ( z ) |^{q( z )} \,\, dv_{g} ( z ) \nonumber \\& \hspace*{0.1cm} - t_{1}^{- \gamma^{-}} \int_{M} g( z ) | \tilde{\mathrm{u}}^{*} ( z ) |^{1 - \gamma ( z )} \,\, dv_{g} ( z ) - \lambda t_{1}^{r^{-} - 1} \int_{M} | \tilde{\mathrm{u}}^{*} ( z ) |^{r( z )} \,\, dv_{g} ( z ).
\end{align}
According to \eqref{10} and \eqref{11} we obtain that
\begin{equation}\label{21}
\lim_{m \rightarrow + \infty} \inf \psi'_{\mathrm{u}_{m}} ( t_{1} ) > \psi'_{\tilde{\mathrm{u}}^{*}} ( t_{1} ) = 0.
\end{equation}
Which implies that, there exists $m_{0} \in \mathbb{N} $ such as $$ \psi'_{\mathrm{u}_{m}} ( t_{1} ) > 0 \,\,\, \mbox{for all} \,\,\, m > m_{0}.$$
Since $ \mathrm{u}_{m} \in \mathcal{N}_{\lambda}^{+} \subseteq \mathcal{N}_{\lambda}$ and $ \psi'_{\mathrm{u}_{m}} = t^{r^{-} - 1} \bigg[ \zeta_{\mathrm{u}_{m}} - \lambda \displaystyle \int_{M} | \mathrm{u}_{m} ( z ) |^{r( z )} \,\, dv_{g} ( z )\bigg],$ we have $$ \psi'_{\mathrm{u}_{m}} ( t ) < 0 \,\, \mbox{for all}\,\, t \in ( 0, 1 ) \,\, \mbox{and} \,\, \psi'_{\mathrm{u}_{m}} ( 1 ) = 0.$$ Then , by \eqref{21} we have $t_{1} > 0.$\\
Then function $ \psi'_{\tilde{\mathrm{u}}^{*}} (\cdot)$ is decreasing on $(0, 1 ).$ Hence, from \eqref{21} we have
\begin{equation}\label{22}
\mathrm{ J}_{\lambda} ( t_{1} \tilde{\mathrm{u}}^{*} ) \leq \mathrm{ J}_{\lambda} ( t \tilde{\mathrm{u}}^{*} ) < \sigma_{\lambda}^{+}.
\end{equation}
However, $t_{1} \tilde{\mathrm{u}}^{*} \in \mathcal{N}_{\lambda}^{+}.$ Hence by \eqref{22} we get 
$$ \sigma_{\lambda}^{+} \leq \mathrm{ J}_{\lambda} ( t_{1} \tilde{\mathrm{u}}^{*} ) < \sigma_{\lambda}^{+},$$
which is a contradiction. Hence $ \mathrm{u}_{m} \rightarrow \tilde{\mathrm{u}}^{*} $ in $W_{0}^{1, q( z )} ( M )$ holds, and we have $$ \mathrm{ J}_{\lambda} ( \mathrm{u}_{m} ) \longrightarrow \mathrm{ J}_{\lambda} ( \tilde{\mathrm{u}}^{*} ),$$
which implies that $$ \mathrm{ J}_{\lambda} ( \tilde{\mathrm{u}}^{*} ) = \sigma_{\lambda}^{+}.$$
Since $ \mathrm{u}_{m} \in \mathcal{N}_{\lambda}^{+} $ for all $ m \in \mathbb{N},$ we have
\begin{align*}
&( p^{+} + \gamma^{-} - 1 ) \int_{M} | D \mathrm{u}_{m} ( z ) |^{p( z )} \,\, dv_{g} ( z ) + ( q^{+} + \gamma^{-} - 1 ) \int_{M} \mu ( z )\, | D \mathrm{u}_{m} ( z ) |^{q( z )} \,\, dv_{g} ( z )\\&  > \lambda ( r^{-} + \gamma^{-} - 1 ) \int_{M} | \mathrm{u}_{m} ( z )|^{r( z )} \,\, dv_{g} ( z ).
\end{align*}
Letting $ m \rightarrow + \infty,$ gives
\begin{align}\label{23}
&( p^{+} + \gamma^{-} - 1 ) \int_{M} | D \tilde{\mathrm{u}}^{*} ( z ) |^{p( z )} \,\, dv_{g} ( z ) + ( q^{+} + \gamma^{-} - 1 ) \int_{M} \mu ( z )\,| D \tilde{\mathrm{u}}^{*} ( z ) |^{q( z )} \,\, dv_{g} ( z ) \nonumber\\&  > \lambda ( r^{-} + \gamma^{-} - 1 ) \int_{M} | \tilde{\mathrm{u}}^{*} ( z )|^{r( z )} \,\, dv_{g} ( z ).
\end{align}
Remind that $ \lambda \in ( 0, \tilde{\lambda}^{*} )$ and $ \tilde{\lambda}^{*} \leq \lambda^{*}.$ Then, by Lemma \ref{lemma2}, we find that equality in \eqref{23} cannot hold. Therefore, we conclude that $\tilde{\mathrm{u}}^{*} ( z ) > 0$ for a.a $ z \in M$ with $ \tilde{\mathrm{u}}^{*} \neq 0.$ This completes the proof of Lemma \ref{lemma4}.
\end{proof}
\begin{lemma}\label{lemma5}
Suppose that hypotheses (i)-(v) hold, suppose $ w \in W_{0}^{1, q( z )} ( M )$ and let $ \lambda \in (0, \tilde{\lambda}^{*}].$ Then, there exists $ \alpha > 0 $ such as for all $ t \in [ 0, \alpha ],$ we have $$\mathrm{ J}_{\lambda} ( \tilde{\mathrm{u}}^{*} ) \leq \mathrm{ J}_{\lambda} ( \tilde{\mathrm{u}}^{*} + t w ).$$
\end{lemma}
\begin{proof}
We introduce the function $\eta_{w} : [0, + \infty ] \longrightarrow \mathbb{R}$ defined by 
\begin{align}
\eta_{w} ( t ) =& \,( p^{+} + \gamma^{-} - 1 ) \int_{M} | D \tilde{\mathrm{u}}^{*} ( z ) + t w( z ) |^{p( z )} \,\, dv_{g} ( z )\nonumber \\& + ( q^{+} + \gamma^{-} - 1 ) \int_{M} \mu ( z ) \, | D \tilde{\mathrm{u}}^{*} ( z ) + t w( z ) |^{q( z )} \,\, dv_{g} ( z )\nonumber \\& - \lambda ( r^{-} + \gamma^{-} - 1 ) \int_{M} | \tilde{\mathrm{u}}^{*} ( z ) + t w( z ) |^{r( z )} \,\, dv_{g} ( z ). 
\end{align}
Since $ \tilde{\mathrm{u}}^{*} \in \mathcal{N}_{\lambda}^{+},$ we have $\eta_{w} ( 0 ) > 0.$ According to the continuity of $\eta_{w} (\cdot )$ we find $ \alpha > 0$ such as $$ \eta_{w} ( t ) > 0 \,\,\, \mbox{for all} \,\,\, t \in [ 0, \alpha ].$$
Thus, $ \tilde{\mathrm{u}}^{*} + t w \in \mathcal{N}_{\lambda}^{+}$ for all $ t \in [ 0, \alpha ].$ Hence, by Lemma \ref{lemma4}, we deduce that
$$ \sigma_{\lambda}^{+} = \mathrm{ J}_{\lambda} ( \tilde{\mathrm{u}}^{*} ) \leq \mathrm{ J}_{\lambda} ( \tilde{\mathrm{u}}^{*} + t w ) \,\,\, \mbox{for all} \,\,\,\, t \in [0, \alpha ].$$
This ends the demonstration.
\end{proof}
The next Lemma proves that $\mathcal{N}_{\lambda}^{+} $ is a natural constraint for the energy functional $\mathrm{ J}_{\lambda},$ see Papageorgiou-R{\u{a}}dulescu-Repov{\v{s}} (\cite{papageorgiou2019nonlinear} p.426 ).
\begin{lemma}\label{lemma6}
Under assumptions (i)-(v), let $ \lambda \in ( 0,  \tilde{\lambda}^{*}].$ Then $\tilde{\mathrm{u}}^{*}$ is a weak solution of $( \mathcal{P} ).$
\end{lemma}
\begin{proof}
Let $ w \in W_{0}^{1, q( z )} ( M ),$ by Lemma \ref{lemma5} we have for all $ t \in [0, \alpha]$ that $$ 0 < \mathrm{ J}_{\lambda} ( \tilde{\mathrm{u}}^{*} + t w ) - \mathrm{ J}_{\lambda} ( \tilde{\mathrm{u}}^{*} ).$$
Then, 
\begin{align}\label{25}
&\frac{1}{1 - \gamma^{-}} \int_{M} g( z ) \big( | \tilde{\mathrm{u}}^{*} ( z ) + t w( z ) |^{1 - \gamma ( z )} - | \tilde{\mathrm{u}}^{*} ( z ) |^{1 - \gamma ( z )} \, \big) \,\, dv_{g} ( z ) \nonumber \\&\leq \frac{1}{p^{-}} \int_{M} \big( | D \tilde{\mathrm{u}}^{*} ( z ) + t D w( z ) |^{p( z )} - | D \tilde{\mathrm{u}}^{*} ( z ) |^{p( z )} \big) \,\, dv_{g} ( z ) \nonumber \\& \hspace*{0.3cm} + \frac{1}{q^{-}} \int_{M} \mu ( z ) \,\big( | D \tilde{\mathrm{u}}^{*} ( z ) + t D w( z ) |^{q( z )} - | D \tilde{\mathrm{u}}^{*} ( z ) |^{q( z )} \big) \,\, dv_{g} ( z ) \nonumber \\& \hspace*{0.3cm}- \frac{\lambda}{q^{-}} \int_{M} \big( | \tilde{\mathrm{u}}^{*} ( z ) + t w( z ) |^{r( z )} - | \tilde{\mathrm{u}}^{*} ( z ) |^{r( z )} \big) \,\, dv_{g} ( z ).
\end{align}
Dividing the above inequality by $t,$ then $ t \rightarrow 0^{+}.$ We deduce that
\begin{align*}
\int_{M} g( z ) ( \tilde{\mathrm{u}}^{*} )^{-\gamma ( z )} w( z ) \,\, dv_{g} ( z ) \leq& \int_{M} | D \tilde{\mathrm{u}}^{*} ( z ) |^{p( z ) - 2} D \tilde{\mathrm{u}}^{*} ( z ) \,.\, D w( z ) \,\, dv_{g} ( z ) \\&+ \int_{M} \mu ( z ) \, |D \tilde{\mathrm{u}}^{*} ( z ) |^{q( z ) - 2} D \tilde{\mathrm{u}}^{*} ( z ) \,.\, D w( z ) \,\, dv_{g} ( z ) \\& - \lambda \int_{M} ( \tilde{\mathrm{u}}^{*} ( z ) )^{r( z ) - 1} \,.\, w \,\, dv_{g} ( z ).
\end{align*}
Since, $ w \in W_{0}^{1, q( z )} ( M )$ is arbitrary. Then that equality must hold, and so $\tilde{\mathrm{u}}^{*} $ is a weak solutions of $ ( \mathcal{P} )$ for all $ \lambda \in ( 0, \tilde{\lambda}^{*} ].$
\end{proof}
Now, using the manifold $ \mathcal{N}_{\lambda}^{-},$ we will achieve a second weak solution when the parameter $\lambda > 0$ is sufficiently small.
\begin{lemma}\label{lemma7}
If hypotheses (i)-(v) are satisfied. Then there exists $\tilde{\lambda}^{*}_{0} \in ( 0, \tilde{\lambda}^{*}]$ such as for all $ \lambda \in (0, \tilde{\lambda}^{*}_{0}]$ we have $\mathrm{ J}_{\lambda} \bigg|_{\mathcal{N}_{\lambda}^{-}} \geq 0.$
\end{lemma}
\begin{proof}
Let $ \mathrm{u} \in \mathcal{N}_{\lambda}^{-}.$ According to Lemma \ref{lemma4}, Theorem \ref{theo3} and the definition of $\mathcal{N}_{\lambda}^{-}$ we obtain that
\begin{align*}
( p^{+} + \gamma^{-} - 1 ) || D \mathrm{u} ||_{p( z )}^{p^{+}} + ( q^{+} + \gamma^{-} - 1 ) \,|| D \mathrm{u} ||_{q( z ), \mu( z )}^{q^{+}} < \lambda \, ( r^{-} + \gamma^{-} - 1) || \mathrm{u} ||_{r( z )}^{r^{-}}.
\end{align*}
Then, $$ \lambda ( r^{-} + \gamma^{-} - 1 ) || \mathrm{u} ||_{r( z )}^{r^{-}} > ( q^{+} + \gamma^{-} - 1 ) \, \mu_{0} \, || D \mathrm{u} ||_{q( z )}^{q^{+}} > \mu_{0} \, ( q^{+} + \gamma^{-} - 1 ) \,. \, c \, || \mathrm{u} ||^{q^{+}}_{q( z )},$$
where, $c$ being the Poincaré canstant. Thus, by Theorem \ref{theo3} we have that $$ || \mathrm{u} ||^{r^{-} - q^{+}}_{r( x )} > \frac{( q^{+} + \gamma^{-} - 1 ) \, \mu_{0} \,.\,c\,.\, c_{8}}{\lambda \, ( r^{-} + \gamma^{-} - 1 )},$$
where, $c_{8} $ being the constant of the embedding Theorem \ref{theo3}. Hence, 
\begin{equation}\label{26}
|| \mathrm{u} ||_{r( z )} > \bigg[ \frac{( q^{+} + \gamma^{-} - 1 ) \, \mu_{0} \,.\, c\,.\,c_{8}}{\lambda \, ( r^{-} + \gamma^{-} - 1 )} \, \bigg]^{\frac{1}{r^{-} - q^{+}}}.
\end{equation}
Arguing by contradiction, suppose that the Lemma is not true. Then, we can find $ \mathrm{u} \in \mathcal{N}_{\lambda}^{-}$ such as $\mathrm{ J}_{\lambda} ( \mathrm{u} ) < 0,$ that is 
\begin{align}\label{27}
&\frac{1}{p^{-}} || D \mathrm{u} ||_{p( z )}^{p^{+}} + \frac{1}{q^{-}} || D \mathrm{u} ||_{q( z ), \mu ( z )}^{q^{+}} - \frac{1}{1 - \gamma^{-}} \int_{M} g( z ) | \mathrm{u} ( z ) |^{1 - \gamma ( z )} \,\, dv_{g} ( z )\nonumber \\& - \frac{\lambda}{r^{+}} || \mathrm{u} ||_{r( z )}^{r^{-}} < 0.
\end{align}
Since $ \mathrm{u} \in \mathcal{N}_{\lambda}^{-} \subseteq \mathcal{N}_{\lambda},$ we know that
\begin{equation}\label{28}
|| D \mathrm{u} ||_{q( z ), \mu ( z )}^{q^{+}} = - || D \mathrm{u} ||_{p( z )}^{p^{+}} + \int_{M} g( z ) | \mathrm{u} ( z ) |^{1 - \gamma ( z )} \,\, dv_{g} ( z ) + \lambda || \mathrm{u} ||_{r( z )}^{r^{-}},
\end{equation}
from \eqref{27} and \eqref{28} we have 
\begin{align*}
&\big( \frac{1}{p^{-}} - \frac{1}{q^{-}} \big) || D \mathrm{u} ||_{p( z )}^{p^{+}} + \big( \frac{1}{q^{-}} - \frac{1}{1 - \gamma^{-}} \big) \int_{M} g( z ) | \mathrm{u} ( z )|^{1 - \gamma ( z )} \,\,dv_{g} ( z )\nonumber \\& + \lambda \, \big( \frac{1}{q^{-}} - \frac{1}{r^{+}} \big) || \mathrm{u} ||_{r( z )}^{r^{-}} < 0,
\end{align*}
which implies that
$$ \lambda \, \big( \frac{1}{q^{-}} - \frac{1}{r^{+}} \big) || \mathrm{u} ||_{r( z )}^{r^{-}} < \big( \frac{1}{1 - \gamma^{-}} - \frac{1}{q^{-}} \big) \, c_{9} \, || \mathrm{u} ||_{r( z )}^{1 - \gamma^{-}} \,\,\, \, \mbox{for some} \,\,\,\, c_{9} > 0.$$
Thus, since $q^{-} < p^{-} < r^{+} $ we have
$$ || \mathrm{u} ||_{r( z )}^{r^{-} + \gamma^{-} - 1} \leq \frac{c_{9} \, ( q^{-} + \gamma^{-} - 1 ) \, r^{+}}{\lambda \, ( 1 - \gamma^{-} ) ( r^{+} - q^{-} )}.$$
Hence, 
\begin{equation}\label{29}
|| \mathrm{u} ||_{r( z )} \leq c_{10} \, \bigg( \frac{1}{\lambda} \bigg)^{\frac{1}{r^{-} + \gamma^{-} - 1}} \,\,\, \mbox{for some} \,\,\,\, c_{10} > 0.
\end{equation}
Applying \eqref{29} in \eqref{26} we get
$$ c_{11} \, \bigg( \frac{1}{\lambda} \bigg)^{\frac{1}{r^{-} - q^{+}}} \leq c_{10} \bigg( \frac{1}{\lambda} \bigg)^{\frac{1}{r^{-} + \gamma^{-} - 1}} \,\,\,\, \mbox{with} \,\,\,\, c_{11} = \bigg( \frac{\mu_{0}\,.\,c \,.\, c_{8}\, ( q^{+} + \gamma^{-} - 1 )}{r^{-} + \gamma^{-} - 1} \bigg)^{\frac{1}{r^{-} - q^{+}}}. $$
Hence, $$ 0 < \frac{c_{11}}{c_{10}} < \lambda^{\frac{1}{r^{-} - q^{+}} - \frac{1}{r^{-} + \gamma^{-} - 1}} = \lambda^{\frac{q^{+} + \gamma^{-} - 1}{( r^{-} - q^{+} ) ( r^{-} + \gamma^{-} - 1 )}} \longrightarrow 0 \,\,\, \mbox{as}\,\, \lambda \rightarrow 0^{+}.$$
Since $1 < q^{+} < r^{-} $ and $\gamma ( \cdot ) \in ( 0, 1 ),$ which is a contradiction.\\
Thus, we conclude that we can find $ \tilde{\lambda}^{*}_{0} \in (0, \tilde{\lambda}^{*} ]$ such as for all $\lambda \in ( 0, \tilde{\lambda}^{*}_{0}]$ we have $$ \mathrm{ J}_{\lambda} \bigg|_{\mathcal{N}_{\lambda}^{-}} \geq 0.$$
\end{proof}
\begin{lemma}\label{lemma8}
Under assumptions (i)-(v), let $\lambda \in ( 0, \tilde{\lambda}^{*}_{0}].$ Then, there exists $\tilde{v}^{*} \in \mathcal{N}_{\lambda}^{-}$ with $ \tilde{v}^{*} \geq 0$ such as $$ \sigma_{\lambda}^{-} = \inf_{\mathcal{N}_{\lambda}^{-}} = \mathrm{ J}_{\lambda} ( \tilde{v}^{*} ) > 0.$$
\end{lemma}
\begin{proof}
Using the same method as Lemma \ref{lemma4}. If $\{ v_{m}\}_{m \in \mathbb{N}} \subseteq \mathcal{N}_{\lambda}^{-}$ is a minimizing sequence, then, by Lemma \ref{lemma1}, we have that $\{ \mathrm{u}_{m} \}_{m \in \mathbb{N}} \subseteq W_{0}^{1, q( z )} ( M )$ is bounded. Then, we may assume that
$$ v_{m} \rightharpoonup \tilde{v}^{*} \,\,\, \mbox{weakly in} \,\,\, W_{0}^{1, q( z )} ( M ) \,\,\, \mbox{and} \,\,\, v_{m} \rightarrow \tilde{v}^{*} \,\,\, \mbox{in} \,\,\, L^{r( z )} ( M ).$$
From \eqref{10} and \eqref{11} we can find $0 < t_{2} $ such as
\begin{equation}\label{30}
\tilde{\zeta'}_{\tilde{v}^{*}} ( t_{2} ) < 0 \,\,\, \mbox{and} \,\,\, \tilde{\zeta}_{\tilde{v}^{*}} ( t_{2} ) = \lambda || \tilde{v}^{*} ||_{r( z )}^{r^{-}}.
\end{equation}
We contend as in the proof of Lemma \ref{lemma4} and using \eqref{30}, we obtain that $ \tilde{v}^{*} \in \mathcal{N}_{\lambda}^{-}, \,\ \tilde{v}^{*} \geq 0, \,\,\, \sigma_{\lambda}^{-} = \mathrm{ J}_{\lambda} ( \tilde{v}^{*} ).$
\end{proof}
\begin{lemma}\label{lemma9}
Under assumptions (i)-(v) and $\lambda \in ( 0, \tilde{\lambda}^{*} ], \,\, \tilde{v}^{*} $ is a weak solution of the problem $( \mathcal{P} )$.
\end{lemma}
\begin{proof}
To demonstrate this Lemma, we use the same reasoning as in the proofs of Lemmas \ref{lemma5} and \ref{lemma6}.
\end{proof}
\section*{Conclusion}
According to the above Lemmas, our problem $( \mathcal{P} )$ has at least two positive solutions $\tilde{\mathrm{u}}^{*}, \, \tilde{v}^{*} \in W_{0}^{1, q( z )} ( M ),$ such as $\mathrm{ J}_{\lambda} ( \tilde{\mathrm{u}}^{*} ) < 0 \leq \mathrm{ J}_{\lambda} ( \tilde{v}^{*} )$ for all $ \lambda \in (0, \, \tilde{\lambda}^{*}_{0}]$ where $\tilde{\lambda}^{*}_{0} > 0.$
\section*{Authors' contributions}
The authors declare that their contributions are equal.
\section*{Acknowledgments} 
This paper has been supported by the RUDN University Strategic Academic Leadership Program and P.R.I.N. 2019.\\
The authors would like to thank the anonymous referees for the valuable suggestions and comments which improved the quality of the presentation.




%
%



\end{document}